\documentclass[11pt]{amsart}
\usepackage{amsmath}
\usepackage[active]{srcltx}
\usepackage{t1enc}
\usepackage[latin2]{inputenc}
\usepackage{verbatim}
\usepackage{amsmath,amsfonts,amssymb,amsthm}
\usepackage[mathcal]{eucal}
\usepackage{enumerate}
\usepackage[centertags]{amsmath}
\usepackage{graphics}
\usepackage[active]{srcltx}

\newtheorem{theorem}{Theorem}

\newtheorem{corollary}{Corollary}

\begin{document}
\author{Ushangi Goginava and Larry Gogoladze}
\title[Convergence in Measure of Strong logarithmic means]{Convergence in
Measure of Strong logarithmic means of double Fourier series}
\address{U. Goginava, Department of Mathematics, Faculty of Exact and
Natural Sciences, Iv. Javakhishvili Tbilisi State University, Chavchavadze
str. 1, Tbilisi 0128, Georgia}
\email{zazagoginava@gmail.com}
\address{L. Gogoladze, Department of Mathematics, Faculty of Exact and
Natural Sciences, Iv. Javakhishvili Tbilisi State University, Chavchavadze
str. 1, Tbilisi 0128, Georgia}
\email{lgogoladze1@hotmail.com}
\maketitle

\begin{abstract}
Nörlund strong logarithmic means of double Fourier series acting from space $%
L\log L\left( \mathbb{T}^{2}\right) $ into space $L_{p}\left( \mathbb{T}%
^{2}\right) ,0<p<1$ are studied. The maximal Orlicz space such that the Nö%
rlund strong logarithmic means of double Fourier series for the functions
from this space converge in two-dimensional measure is found.
\end{abstract}

\footnotetext{%
2010 Mathematics Subject Classification 42A24 .
}

\section{Introduction}

The rectangular partial sums of double Fourier series $S_{n,m}\left(
f;x,y\right) $ of the function $f\in L_{p}\left( \mathbb{T}^{2}\right) ,%
\mathbb{T}:=[-\pi ,\pi ),1<p<\infty $ converge in $L_{p}$ norm to the
function $f$, as $n\rightarrow \infty $ \cite{Zh}. In the case $L_{1}\left(
\mathbb{T}^{2}\right) $ this result does not hold . But for $f\in
L_{1}\left( \mathbb{T}\right) $, the operator $S_{n}\left( f;x\right) $ are
of weak type (1,1) \cite{Zy}. This estimate implies convergence of $%
S_{n}\left( f;x\right) $ in measure on $\mathbb{T}$ to the function $f\in
L_{1}\left( \mathbb{T}\right) $. However, for double Fourier series this
result does not hold \cite{Ge, Kon}. \ Moreover, it is proved that
quadratical partial sums $S_{n,n}\left( f;x,y\right) $ of double Fourier
series do not converge in two-dimensional measure on $\mathbb{T}^{2}$ \ even
for functions from Orlicz spaces wider than Orlicz space $L\log L\left(
\mathbb{T}^{2}\right) $. On the other hand, it is well-known that if the
function $f\in L\log L\left( \mathbb{T}^{2}\right) $, then rectangular
partial sums $S_{n,m}\left( f;x,y\right) $ converge in measure on $\mathbb{T}%
^{2}$.

Classical regular summation methods often improve the convergence of Fourier
seeries. For instance, the Fej\'er means of the double Fourier series of the
function $f\in L_{1}\left( \mathbb{T}^{2}\right) $ converge in $L_{1}\left(
\mathbb{T}^{2}\right) $ norm to the function $f$ \cite{Zh}. These means
present the particular case of the N\"{o}rlund means.

The Nörlund logarithmic means of double Fourier series are defined by%
\begin{equation*}
t_{n,m}\left( f;x,y\right) :=\frac{1}{l_{n}l_{m}}\sum\limits_{i=0}^{n}\sum%
\limits_{j=0}^{m}\frac{S_{i,j}\left( f;x,y\right) }{\left( n-i+1\right)
\left( m-j+1\right) },
\end{equation*}%
where $l_{n}:=\sum_{k=1}^{n+1}\left( 1/k\right) $ and by $S_{i,j}\left(
f;x,y\right) $ we denote rectangular partial sums of double Fourier series
of the function $f$.

It is well know that the method of Nörlund logarithmic means of double
Fourier series, is weaker than the Cesáro method of any positive order. In
\cite{Tk1} Tkebuchava proved, that these means of double Fourier series in
general do not converge in two-dimensional measure on $\mathbb{T}^{2}$ even
for functions from Orlicz spaces wider than Orlicz space $L\log L\left(
\mathbb{T}^{2}\right) $. For logarithmic means $t_{n,m}\left( f;x,y\right) $
of double Fourier series Tkebuchava \cite{Tk} proved that the following
results are true.

\begin{theorem}
\label{tkebuchava}Let $L_{Q}\left( \mathbb{T}^{2}\right) $ be an Orlicz
space, such that%
\begin{equation*}
L_{Q}\left( \mathbb{T}^{2}\right) \nsubseteq L\log L\left( \mathbb{T}%
^{2}\right) .
\end{equation*}%
Then the set of the function from the Orlicz space $L_{Q}\left( \mathbb{T}%
^{2}\right) $ with logarithmic means of rectangular partial sums of double
Fourier series, convergent in measure on $\mathbb{T}^{2}$, is of first Baire
category in $L_{Q}\left( \mathbb{T}^{2}\right) $.
\end{theorem}

On the other hand, it is noted, that the regularuty of summation method
 does not allow to deduce the summability in measure of functional
sequence from its convergence in measure (see \cite{GGT}, Remark 1).

In this paper we consider the strong logarithmic means of rectangular
partial sums double Fourier series and prove that these means are acting
from space $L\log L\left( \mathbb{T}^{2}\right) $ into space $L_{p}\left(
\mathbb{T}^{2}\right) ,0<p<1$ (see Theorem \ref{logest} ). This fact implies
the convergence of strong logarithmic means of rectangular partial sums of
double Fourier series in measure on $\mathbb{T}^{2}$ to the function $f\in
L\log L\left( T^{2}\right) $ (see Corollary \ref{cor} ). Uniting these
results with statement from \cite{Tk1} we obtain, that the rectangular
partial sums of double Fourier series converge in measure for all functions
from Orlicz space if and only if their strong N\"{o}rlund logarithmic means and
strong N\"{o}rlund logarithmic means converge in measure (see Theorem \ref{eq}).
Thus, not all classic regular summation methods can improve the convergence
in measure of double Fourier series.

The results for summability of logarithmic means of Walsh-Fourier series can
be found in \cite{GogJAT, GGNSMH,gg,sza,ya}.

\section{Double Fourier Series and Conjugate Functions}

We denote by $L_{0}=L_{0}(\mathbb{T}^{2})$ the Lebesque space of functions
that are measurable and finite almost everywhere on $\mathbb{T}^{2}$.

Let $L_{Q}=L_{Q}(\mathbb{T}^{2})$ be the Orlicz space \cite{K-R} generated
by Young function $Q$, i.e. $Q$ is a convex continuous even function such
that $Q(0)=0$ and

\begin{equation*}
\lim\limits_{u\rightarrow +\infty }\frac{Q\left( u\right) }{u}=+\infty
,\,\,\,\,\lim\limits_{u\rightarrow 0}\frac{Q\left( u\right) }{u}=0.
\end{equation*}

This space is endowed with the norm
\begin{equation*}
\Vert f\Vert _{L_{Q}(\mathbb{T}^{2})}=\inf \{k>0:\iint\limits_{\mathbb{T}%
^{2}}Q(\left\vert f(x,y)\right\vert /k)dxdy\leq 1\}.
\end{equation*}

In particular, if $Q(u)=u\log ^{+}u$ $,\log ^{+}u:=1_{\left\{ u>1\right\}
}\log u$, then the corresponding space will be denoted by $L\log L(\mathbb{T}%
^{2})$.

Given a function $f\in L_{1}\left( \mathbb{T}^{2}\right) $, its double
Fourier series is defined by
\begin{equation}
\sum_{\left( n,m\right) \in \mathbb{Z}^{2}}\widehat{f}\left( m,n\right)
e^{imx}e^{iny},  \label{FS}
\end{equation}%
where $\mathbb{Z}:=\left\{ ...,-1,0,1,2,...\right\} $ and
\begin{equation}
\widehat{f}\left( m,n\right) =\frac{1}{4\pi ^{2}}\iint\limits_{\mathbb{T}%
^{2}}f(x,y)e^{-imx}e^{-iny}dxdy  \label{coeff}
\end{equation}%
are the Fourier coefficients of the function $f$ .

Denote by $S_{n,m}\left( f;x,y\right) $ the $\left( n,m\right) $th symmetric
rectangular partial sums of series (\ref{FS}). As is well-known, we have%
\begin{equation*}
S_{n,m}\left( f;x,y\right) =\frac{1}{\pi ^{2}}\iint\limits_{\mathbb{T}%
^{2}}f\left( s,t\right) D_{n}\left( x-s\right) D_{m}\left( y-t\right) dsdt
\end{equation*}%
and%
\begin{equation*}
D_{n}\left( u\right) :=\frac{\sin \left( \left( n+1/2\right) u\right) }{%
2\sin \left( u/2\right) }
\end{equation*}%
is the Dirichlet kernel.

One can associate three conjugate series to the double Fourier series (\ref%
{FS}):%
\begin{equation}
\widetilde{f}^{\left( 1,0\right) }\sim \sum\limits_{\left( j,k\right) \in
\mathbb{Z}^{2}}\left( -i\text{sign}j\right) \widehat{f}\left( j,k\right)
e^{i\left( jx+ky\right) }  \label{CF1}
\end{equation}%
(conjugate with respect to the first variable),%
\begin{equation}
\widetilde{f}^{\left( 0,1\right) }\sim \sum\limits_{\left( j,k\right) \in
\mathbb{Z}^{2}}\left( -i\text{sign}k\right) \widehat{f}\left( j,k\right)
e^{i\left( jx+ky\right) }  \label{CF2}
\end{equation}%
(conjugate with respect to the second variable)%
\begin{equation}
\widetilde{f}^{\left( 1,1\right) }\sim \sum\limits_{\left( j,k\right) \in
\mathbb{Z}^{2}}\left( -i\text{sign}j\right) \left( -i\text{sign}k\right)
\widehat{f}\left( j,k\right) e^{i\left( jx+ky\right) }  \label{CF11}
\end{equation}%
(conjugate with respect to both variables).

As is well known, if $f$ is an integrable function then%
\begin{equation*}
\widetilde{f}^{\left( 1,0\right) }\left( x,y\right) =\text{p.v.}\frac{1}{\pi
}\int\limits_{\mathbb{T}}\frac{f\left( s,y\right) }{2\tan \left( \frac{x-s}{2%
}\right) }ds,
\end{equation*}%
\begin{equation*}
\widetilde{f}^{\left( 0,1\right) }\left( x,y\right) =\text{p.v.}\frac{1}{\pi
}\int\limits_{\mathbb{T}}\frac{f\left( x,t\right) }{2\tan \left( \frac{y-t}{2%
}\right) }dt
\end{equation*}%
and%
\begin{equation*}
\widetilde{f}^{\left( 1,1\right) }\left( x,y\right) =\text{p.v.}\frac{1}{\pi
^{2}}\iint\limits_{\mathbb{T}^{2}}\frac{f\left( s,t\right) }{2\tan \left(
\frac{x-s}{2}\right) 2\tan \left( \frac{y-t}{2}\right) }dsdt.
\end{equation*}

Privalov's theorem (see e.g. \cite{Zy}, vol. II, p. 121) immediately implies
the a. e. existence of $\widetilde{f}^{\left( 1,0\right) }$ and $\widetilde{f%
}^{\left( 0,1\right) }$ under the assumption $f\in L_{1}\left( \mathbb{T}%
^{2}\right) $. \ The a. e. existence of $\widetilde{f}^{\left( 1,1\right) }$
for $f\in L\log L\left( \mathbb{T}^{2}\right) $ was proved by Zygmund \cite%
{Zy2,Zh2}.

We shall consider the symmetric rectangular partial sums of series (\ref{CF1}%
)-(\ref{CF11}) defined by%
\begin{equation*}
\widetilde{S}_{n,m}^{10}\left( f;x,y\right) :=\sum\limits_{|j|\leq
n}\sum\limits_{|k|\leq m}\left( -i\text{sign}j\right) \widehat{f}\left(
j,k\right) e^{i\left( jx+ky\right) },
\end{equation*}%
\begin{equation*}
\widetilde{S}_{n,m}^{01}\left( f;x,y\right) :=\sum\limits_{|j|\leq
n}\sum\limits_{|k|\leq m}\left( -i\text{sign}k\right) \widehat{f}\left(
j,k\right) e^{i\left( jx+ky\right) }
\end{equation*}%
and%
\begin{equation*}
\widetilde{S}_{n,m}^{11}\left( f;x,y\right) :=\sum\limits_{|j|\leq
n}\sum\limits_{|k|\leq m}\left( -i\text{sign}j\right) \left( -i\text{sign}%
k\right) \widehat{f}\left( j,k\right) e^{i\left( jx+ky\right) }.
\end{equation*}

It follows from (\ref{coeff}) that%
\begin{equation*}
\widetilde{S}_{n,m}^{10}\left( f;x,y\right) =\frac{1}{\pi ^{2}}\iint\limits_{%
\mathbb{T}^{2}}f\left( s,t\right) \widetilde{D}_{n}\left( x-s\right)
D_{m}\left( y-t\right) dsdt,
\end{equation*}%
\begin{equation*}
\widetilde{S}_{n,m}^{01}\left( f;x,y\right) =\frac{1}{\pi ^{2}}\iint\limits_{%
\mathbb{T}^{2}}f\left( s,t\right) D_{n}\left( x-s\right) \widetilde{D}%
_{m}\left( y-t\right) dsdt
\end{equation*}%
and%
\begin{equation*}
\widetilde{S}_{n,m}^{11}\left( f;x,y\right) =\frac{1}{\pi ^{2}}\iint\limits_{%
\mathbb{T}^{2}}f\left( s,t\right) \widetilde{D}_{n}\left( x-s\right)
\widetilde{D}_{m}\left( y-t\right) dsdt,
\end{equation*}%
where%
\begin{equation}
\widetilde{D}_{m}\left( u\right) :=\frac{1}{2\tan \left( u/2\right) }-\frac{%
\cos \left( \left( m+1\right) u\right) }{2\sin \left( u/2\right) },m=1,2,...
\label{conjdir}
\end{equation}%
is the conjugate Dirichlet kernel.

In this paper we also consider the following operators%
\begin{equation*}
\widetilde{\overline{S}}_{n,m}^{10}\left( f;x,y\right) =\frac{1}{\pi ^{2}}%
\iint\limits_{\mathbb{T}^{2}}f\left( s,t\right) \widetilde{D}_{n}\left(
x-s\right) \overline{D}_{m}\left( y-t\right) dsdt,
\end{equation*}%
\begin{equation*}
\widetilde{\overline{S}}_{n,m}^{01}\left( f;x,y\right) =\frac{1}{\pi ^{2}}%
\iint\limits_{\mathbb{T}^{2}}f\left( s,t\right) \overline{D}_{n}\left(
x-s\right) \widetilde{D}_{m}\left( y-t\right) dsdt
\end{equation*}%
and%
\begin{equation*}
\overline{S}_{n,m}\left( f;x,y\right) =\frac{1}{\pi ^{2}}\iint\limits_{%
\mathbb{T}^{2}}f\left( s,t\right) \overline{D}_{n}\left( x-s\right)
\overline{D}_{m}\left( y-t\right) dsdt,
\end{equation*}%
where $\overline{D}_{n}\left( u\right) $ is a modified Dirichlet kernel
defined by%
\begin{equation*}
\overline{D}_{n}\left( u\right) :=\frac{\sin \left( nu\right) }{2\tan \left(
u/2\right) }.
\end{equation*}

\section{Strong Riesz Logarithmic and Strong N\"{o}rlund Logarithmic means}

The strong Riesz logarithmic means, strong Nörlund logarithmic means and
strong Fejér means of rectangular partial sums $\widetilde{S}_{i,j}^{ab}f$
defined by%
\begin{equation*}
\widetilde{R}_{n,m}^{ab}\left( f;x,y\right) :=\frac{1}{l_{n}l_{m}}%
\sum\limits_{i=0}^{n}\sum\limits_{j=0}^{m}\frac{\left\vert \widetilde{S}%
_{i,j}^{ab}\left( f;x,y\right) \right\vert }{\left( i+1\right) \left(
j+1\right) },
\end{equation*}%
\begin{equation*}
\widetilde{\tau }_{n,m}^{ab}\left( f;x,y\right) :=\frac{1}{l_{n}l_{m}}%
\sum\limits_{i=0}^{n}\sum\limits_{j=0}^{m}\frac{\left\vert \widetilde{S}%
_{i,j}^{ab}\left( f;x,y\right) \right\vert }{\left( n-i+1\right) \left(
m-j+1\right) },
\end{equation*}%
\begin{equation*}
\widetilde{\sigma }_{n,m}^{ab}\left( f;x,y\right) :=\frac{1}{\left(
n+1\right) \left( m+1\right) }\sum\limits_{i=0}^{n}\sum\limits_{j=0}^{m}%
\left\vert \widetilde{S}_{i,j}^{ab}\left( f;x\right) \right\vert ,a,b=0,1.
\end{equation*}

Denote%
\begin{equation*}
\widetilde{R}_{n,m}^{00}\left( f\right) =R_{n,m}\left( f\right) ,\widetilde{S%
}_{n,m}^{00}\left( f\right) =S_{n,m}\left( f\right) ,\widetilde{\tau }%
_{n,m}^{00}\left( f\right) =\tau _{n,m}\left( f\right) ,\widetilde{\sigma }%
_{n,m}^{00}\left( f\right) =\sigma _{n,m}\left( f\right) .
\end{equation*}

In \cite{Gogoladze}, in particular, it is proved that the following
estimation is true.

\begin{theorem}
\label{Gogoladze}Let $f\in L\log L\left( \mathbb{T}^{2}\right) $ and $0<p<1$%
. Then for any $a,b=0,1$ the following estimation holds%
\begin{equation*}
\left( \iint\limits_{\mathbb{T}^{2}}\left( \sup\limits_{n,m}\widetilde{%
\sigma }_{n,m}^{\left( a,b\right) }\left( f;x,y\right) \right)
^{p}dxdy\right) ^{1/p}\leq c_{1}\iint\limits_{\mathbb{T}}\left\vert f\left(
x,y\right) \right\vert \log ^{+}\left\vert f\left( x,y\right) \right\vert
dxdy+c_{2}.
\end{equation*}
\end{theorem}

Appling Hardy's transformation, we obtain%
\begin{eqnarray}
&&l_{n}l_{m}\widetilde{R}_{n,m}^{ab}\left( f;x,y\right)  \label{L1-L4} \\
&=&\sum\limits_{i=0}^{n-1}\sum\limits_{j=0}^{m-1}\frac{\widetilde{\sigma }%
_{i,j}^{ab}\left( f;x,y\right) }{\left( i+2\right) \left( j+2\right) }
\notag \\
&&+\sum\limits_{j=0}^{m-1}\frac{1}{j+2}\widetilde{\sigma }_{n,j}^{ab}\left(
f;x,y\right)  \notag \\
&&+\sum\limits_{i=0}^{n-1}\frac{1}{i+2}\widetilde{\sigma }_{i,m}^{ab}\left(
f;x,y\right)  \notag \\
&&+\widetilde{\sigma }_{n,m}^{ab}\left( f;x,y\right) .  \notag
\end{eqnarray}%
Consequently, from Theorem \ref{Gogoladze} we obtain%
\begin{eqnarray}
&&\left( \iint\limits_{\mathbb{T}^{2}}\left( \widetilde{R}_{n,m}^{ab}\left(
f;x,y\right) \right) ^{p}dxdy\right) ^{1/p}  \label{weak-R} \\
&\leq &4\left( \iint\limits_{\mathbb{T}^{2}}\left( \sup\limits_{n,m}%
\widetilde{\sigma }_{n,m}^{\left( a,b\right) }\left( f;x,y\right) \right)
^{p}dxdy\right) ^{1/p}  \notag \\
&\leq &c_{1}\iint\limits_{\mathbb{T}^{2}}\left\vert f\left( x,y\right)
\right\vert \log ^{+}\left\vert f\left( x,y\right) \right\vert dxdy+c_{2}%
\text{ \ }\left( f\in L\log L\left( \mathbb{T}^{2}\right) \right) .  \notag
\end{eqnarray}

Since%
\begin{eqnarray*}
&&\left( \int\limits_{\mathbb{T}}\left( \sup\limits_{n}\sigma _{n}\left(
f;x\right) \right) ^{p}dx\right) ^{1/p},\left( \int\limits_{\mathbb{T}%
}\left( \sup\limits_{n}\widetilde{\sigma }_{n}\left( f;x\right) \right)
^{p}dx\right) ^{1/p} \\
&\leq &c_{1}\int\limits_{\mathbb{T}}\left\vert f\left( x\right) \right\vert
dx,f\in L_{1}\left( \mathbb{T}\right) ,0<p<1.
\end{eqnarray*}

analogously, for one dimensional case we can prove that%
\begin{eqnarray}
&&\left( \int\limits_{\mathbb{T}}\left( R_{n}\left( f;x\right) \right)
^{p}dx\right) ^{1/p},\left( \int\limits_{\mathbb{T}}\left( \widetilde{R}%
_{n}\left( f;x\right) \right) ^{p}dx\right) ^{1/p}  \label{estoneconj} \\
&\leq &c_{1}\int\limits_{\mathbb{T}}\left\vert f\left( x\right) \right\vert
dx,f\in L_{1}\left( \mathbb{T}\right) ,0<p<1,  \notag
\end{eqnarray}%
where $\sigma _{n}\left( f;x\right) ,\widetilde{\sigma }_{n}\left(
f;x\right) ,R_{n}\left( f;x\right) $ and $\widetilde{R}_{n}\left( f;x\right)
$ are strong Fej\'er and strong Riesz means of Fourier series and conjugate
Fourier series.

\section{Main Results}

\begin{theorem}
\label{logest}Let $f\in L\log L\left( \mathbb{T}^{2}\right) $ and $0<p<1$.
Then the following estimation holds%
\begin{equation*}
\left( \iint\limits_{\mathbb{T}^{2}}\left( \tau _{n,m}\left( f;x,y\right)
\right) ^{p}dxdy\right) ^{1/p}\leq c_{1}\iint\limits_{\mathbb{T}%
^{2}}\left\vert f\left( x,y\right) \right\vert \log ^{+}\left\vert f\left(
x,y\right) \right\vert dxdy+c_{2}.
\end{equation*}
\end{theorem}

\begin{theorem}
\label{theoremconv}Let $f\in L\log L\left( \mathbb{T}^{2}\right) $ and $%
0<p<1 $. Then%
\begin{equation*}
\iint\limits_{\mathbb{T}^{2}}\left( \frac{1}{l_{n}l_{m}}\sum%
\limits_{i=0}^{n}\sum\limits_{j=0}^{m}\frac{\left\vert S_{i,j}\left(
f;x,y\right) -f\left( x,y\right) \right\vert }{\left( n-i+1\right) \left(
m-j+1\right) }\right) ^{p}dxdy\rightarrow 0
\end{equation*}%
as $n,m\rightarrow \infty .$
\end{theorem}

\begin{corollary}
\label{cor}Let $f\in L\log L\left( \mathbb{T}^{2}\right) $. Then%
\begin{equation*}
\frac{1}{l_{n}l_{m}}\sum\limits_{i=0}^{n}\sum\limits_{j=0}^{m}\frac{%
\left\vert S_{i,j}\left( f;x,y\right) -f\left( x,y\right) \right\vert }{%
\left( n-i+1\right) \left( m-j+1\right) }\rightarrow 0
\end{equation*}%
in measure on $\mathbb{T}^{2}$, as $n,m\rightarrow \infty .$
\end{corollary}

Uniting these results with statement  \ref{tkebuchava} of Tkebuchava we obtain.

\begin{theorem}
\label{eq}The following conditions are equalent

a)%
\begin{equation*}
L_{Q}(\mathbb{T}^{2})\subset L\log L(\mathbb{T}^{2});
\end{equation*}%
\newline
b) the strong N\"{o}rlund logarithmic means of double Fourier series for all
functions from Orlicz space $L_{Q}(\mathbb{T}^{2})$ converges in measure on $%
\mathbb{T}^{2}$;\newline
\newline
c) the N\"{o}rlund logarithmic means of double Fourier series for all functions
from Orlicz space $L_{Q}(\mathbb{T}^{2})$ converges in measure on $\mathbb{T}%
^{2}$;\newline
\newline
\end{theorem}

\section{Proof of Main Results}

\begin{proof}[Proof of Theorem \protect\ref{logest}]
Set $\alpha _{n}\left( t\right) :=\sin \left( \left( n+1\right) t\right)
,\beta _{n}\left( t\right) :=\cos \left( \left( n+1\right) t\right) $. Then
we can write%
\begin{eqnarray}
&&S_{n-k}\left( f;x\right)   \label{S_n-k} \\
&=&\frac{1}{\pi }\int\limits_{\mathbb{T}}f\left( t\right) \frac{\sin \left(
\left( n-k+1/2\right) \left( x-t\right) \right) }{2\sin \left( \left(
x-t\right) /2\right) }dt  \notag \\
&=&\frac{1}{\pi }\int\limits_{\mathbb{T}}f\left( t\right) \sin \left( \left(
n+1\right) \left( x-t\right) \right) \frac{\cos \left( \left( k+1/2\right)
\left( x-t\right) \right) }{2\sin \left( \left( x-t\right) /2\right) }dt
\notag \\
&&-\frac{1}{\pi }\int\limits_{\mathbb{T}}f\left( t\right) \cos \left( \left(
n+1\right) \left( x-t\right) \right) \frac{\sin \left( \left( k+1/2\right)
\left( x-t\right) \right) }{2\sin \left( \left( x-t\right) /2\right) }dt
\notag \\
&=&\frac{1}{\pi }\int\limits_{\mathbb{T}}f\left( t\right) \sin \left( \left(
n+1\right) \left( x-t\right) \right)   \notag \\
&&\times \left( \frac{\cos \left( \left( k+1/2\right) \left( x-t\right)
\right) }{2\sin \left( \left( x-t\right) /2\right) }-\frac{\cos \left(
\left( x-t\right) /2\right) }{2\sin \left( \left( x-t\right) /2\right) }%
\right) dt  \notag
\end{eqnarray}%
\begin{eqnarray*}
&&+\frac{1}{\pi }\int\limits_{\mathbb{T}}f\left( t\right) \frac{\sin \left(
\left( n+1\right) \left( x-t\right) \right) }{2\tan \left( \left( x-t\right)
/2\right) }dt \\
&&-\frac{1}{\pi }\int\limits_{\mathbb{T}}f\left( t\right) \cos \left( \left(
n+1\right) \left( x-t\right) \right) \frac{\sin \left( \left( k+1/2\right)
\left( x-t\right) \right) }{2\sin \left( \left( x-t\right) /2\right) }dt \\
&=&-\frac{\alpha _{n}\left( x\right) }{\pi }\int\limits_{\mathbb{T}}f\left(
t\right) \beta _{n}\left( t\right) \widetilde{D}_{k}\left( x-t\right) dt \\
&&+\frac{\beta _{n}\left( x\right) }{\pi }\int\limits_{\mathbb{T}}f\left(
t\right) \alpha _{n}\left( t\right) \widetilde{D}_{k}\left( x-t\right) dt
\end{eqnarray*}%
\begin{eqnarray*}
&&+\frac{1}{\pi }\int\limits_{\mathbb{T}}f\left( t\right) \frac{\sin \left(
\left( n+1\right) \left( x-t\right) \right) }{2\tan \left( \left( x-t\right)
/2\right) }dt \\
&&-\frac{\beta _{n}\left( x\right) }{\pi }\int\limits_{\mathbb{T}}f\left(
t\right) \beta _{n}\left( t\right) \frac{\sin \left( \left( k+1/2\right)
\left( x-t\right) \right) }{2\sin \left( \left( x-t\right) /2\right) }dt \\
&&-\frac{\alpha _{n}\left( x\right) }{\pi }\int\limits_{\mathbb{T}}f\left(
t\right) \alpha _{n}\left( t\right) \frac{\sin \left( \left( k+1/2\right)
\left( x-t\right) \right) }{2\sin \left( \left( x-t\right) /2\right) }dt \\
&=&-\alpha _{n}\left( x\right) \widetilde{S}_{k}\left( f\beta _{n};x\right)
+\beta _{n}\left( x\right) \widetilde{S}_{k}\left( f\alpha _{n};x\right)  \\
&&-\beta _{n}\left( x\right) S_{k}\left( f\beta _{n};x\right) -\alpha
_{n}\left( x\right) S_{k}\left( f\alpha _{n};x\right) +\overline{S}%
_{n+1}\left( f;x\right)
\end{eqnarray*}%
Hence%
\begin{eqnarray*}
\tau _{n}\left( f;x\right)  &:&=\frac{1}{l_{n}}\sum\limits_{k=0}^{n}\frac{%
\left\vert S_{k}\left( f;x\right) \right\vert }{n-k+1}\leq \widetilde{R}%
_{n}\left( f\beta _{n},x\right) +\widetilde{R}_{n}\left( f\alpha
_{n},x\right)  \\
&&+R_{n}\left( f\beta _{n},x\right) +R\sigma _{n}\left( f\alpha
_{n},x\right) +\overline{S}_{n+1}\left( f;x\right) .
\end{eqnarray*}

Since%
\begin{equation*}
\left( \int\limits_{\mathbb{T}}\left\vert S_{n}\left( f;x\right) \right\vert
^{p}dx\right) ^{1/p}\leq c_{p}\int\limits_{\mathbb{T}}\left\vert f\left(
x\right) \right\vert dx.
\end{equation*}

from (\ref{estoneconj}) we conclude that%
\begin{equation}
\left( \int\limits_{\mathbb{T}}\left( \tau _{n}\left( f;x\right) \right)
^{p}dx\right) ^{1/p}\leq c_{p}\int\limits_{\mathbb{T}}\left\vert f\left(
x\right) \right\vert dx,f\in L_{1}\left( \mathbb{T}\right) ,0<p<1,f\in
L_{1}\left( \mathbb{T}\right) .  \label{t_n}
\end{equation}

Now, we consider rectangular partial sums of double Fourier series. From (%
\ref{S_n-k}) we can write
\begin{eqnarray}
S_{n-i,m-j}\left( f;x,y\right) &=&S_{n-i}\left( S_{m-j}\left( f;y\right)
;x\right)  \label{S_n-im-j} \\
&&=-\alpha _{n}\left( x\right) \widetilde{S}_{i}\left( S_{m-j}\left(
f;y\right) \beta _{n};x\right)  \notag \\
&&+\beta _{n}\left( x\right) \widetilde{S}_{i}\left( S_{m-j}\left(
f;y\right) \alpha _{n};x\right)  \notag \\
&&-\beta _{n}\left( x\right) S_{i}\left( S_{m-j}\left( f;y\right) \beta
_{n};x\right)  \notag \\
&&-\alpha _{n}\left( x\right) S_{i}\left( S_{m-j}\left( f;y\right) \alpha
_{n};x\right)  \notag \\
&&+\overline{S}_{n+1}\left( S_{m-j}\left( f,y\right) ;x\right)  \notag \\
&:&=\sum\limits_{s=1}^{4}I_{s}\left( i,j;x,y\right) +\overline{S}%
_{n+1}\left( S_{m-j}\left( f,y\right) ;x\right) .  \notag
\end{eqnarray}

Now, we turn our attention to $I_{1}\left( i,j;x,y\right) $. From (\ref%
{S_n-k}) we have%
\begin{eqnarray}
I_{1}\left( i,j;x,y\right) &=&-\alpha _{n}\left( x\right) S_{m-j}\left(
\widetilde{S}_{i}\left( f\beta _{n};x\right) ;y\right)  \label{I1-I5} \\
&=&\alpha _{n}\left( x\right) \alpha _{m}\left( y\right) \widetilde{S}%
_{j}\left( \widetilde{S}_{i}\left( f\beta _{n};x\right) \beta _{m};y\right)
\notag \\
&&-\alpha _{n}\left( x\right) \beta _{m}\left( y\right) \widetilde{S}%
_{j}\left( \widetilde{S}_{i}\left( f\beta _{n};x\right) \alpha _{m};y\right)
\notag \\
&&+\alpha _{n}\left( x\right) \beta _{m}\left( y\right) S_{j}\left(
\widetilde{S}_{i}\left( f\beta _{n};x\right) \beta _{m};y\right)  \notag \\
&&+\alpha _{n}\left( x\right) \alpha _{m}\left( y\right) S_{j}\left(
\widetilde{S}_{i}\left( f\beta _{n};x\right) \alpha _{m};y\right)  \notag \\
&&-\alpha _{n}\left( x\right) \overline{S}_{m+1}\left( \widetilde{S}%
_{i}\left( f\beta _{n};x\right) ;y\right)  \notag \\
&=&\alpha _{n}\left( x\right) \alpha _{m}\left( y\right) \widetilde{S}%
_{ij}^{11}\left( f\beta _{n}\beta _{m};x,y\right)  \notag \\
&&-\alpha _{n}\left( x\right) \beta _{m}\left( y\right) \widetilde{S}%
_{ij}^{11}\left( f\beta _{n}\alpha _{m};x,y\right)  \notag \\
&&+\alpha _{n}\left( x\right) \beta _{m}\left( y\right) \widetilde{S}%
_{ij}^{10}\left( f\beta _{n}\beta _{m};x,y\right)  \notag \\
&&+\alpha _{n}\left( x\right) \alpha _{m}\left( y\right) \widetilde{S}%
_{ij}^{10}\left( f\beta _{n}\alpha _{m};x,y\right)  \notag \\
&&-\alpha _{n}\left( x\right) \widetilde{\overline{S}}_{i,m+1}^{01}\left(
f\beta _{n};x,y\right)  \notag \\
&=&\sum\limits_{l=1}^{4}I_{1l}\left( i,j;x,y\right) +I_{15}\left(
i,m;x,y\right) .  \notag
\end{eqnarray}

From (\ref{weak-R}) we have%
\begin{eqnarray}
&&\iint\limits_{\mathbb{T}^{2}}\left( \frac{1}{l_{n}l_{m}}%
\sum\limits_{i=0}^{n}\sum\limits_{j=0}^{m}\frac{\left\vert I_{11}\left(
i,j;x,y\right) \right\vert }{\left( i+1\right) \left( j+1\right) }\right)
^{p}dxdy  \label{I11} \\
&\leq &\iint\limits_{\mathbb{T}^{2}}\left\vert \widetilde{R}%
_{n,m}^{11}\left( f\beta _{n}\beta _{m};x,y\right) \right\vert ^{p}dxdy
\notag
\end{eqnarray}%
\begin{equation*}
\leq c_{1}\iint\limits_{\mathbb{T}^{2}}\left\vert f\left( x,y\right)
\right\vert \log ^{+}\left\vert f\left( x,y\right) \right\vert dxdy+c_{2}.
\end{equation*}

Analogously, we can prove that%
\begin{eqnarray}
&&\iint\limits_{\mathbb{T}^{2}}\left( \frac{1}{l_{n}l_{m}}%
\sum\limits_{i=0}^{n}\sum\limits_{j=0}^{m}\frac{\left\vert I_{1l}\left(
i,j;x,y\right) \right\vert }{\left( i+1\right) \left( j+1\right) }\right)
^{p}dxdy  \label{Il1} \\
&\leq &c_{1}\iint\limits_{\mathbb{T}^{2}}\left\vert f\left( x,y\right)
\right\vert \log ^{+}\left\vert f\left( x,y\right) \right\vert
dxdy+c_{2},l=2,3,4.  \notag
\end{eqnarray}

Now, we turn our attention to $I_{15}\left( i,m;x,y\right) $. Since%
\begin{equation*}
\widetilde{\overline{S}}_{i,m+1}^{10}\left( f\beta _{n};x,y\right) =%
\widetilde{S}_{i}\left( \overline{S}_{m+1}\left( f;y\right) \beta
_{n};x\right) ,
\end{equation*}%
\begin{equation*}
f\left( \cdot ,y\right) \in L\log L\left( \mathbb{T}\right) \text{, for a.e.
}y\in \mathbb{T}\text{ and }f\in L\log L\left( \mathbb{T}^{2}\right)
\end{equation*}%
and%
\begin{equation*}
\int\limits_{\mathbb{T}}\left\vert \overline{S}_{m+1}\left( f;x,y\right)
\right\vert dx\leq c_{1}\int\limits_{\mathbb{T}}\left\vert f\left(
x,y\right) \right\vert \log ^{+}\left\vert f\left( x,y\right) \right\vert
dx+c_{2}
\end{equation*}%
from (\ref{estoneconj}) we obtain%
\begin{eqnarray*}
&&\left( \int\limits_{\mathbb{T}}\left( \frac{1}{l_{n}}\sum\limits_{i=0}^{n}%
\frac{\left\vert \widetilde{S}_{i,m+1}^{10}\left( f\beta _{n};x,y\right)
\right\vert }{i+1}\right) ^{p}dx\right) ^{1/p} \\
&\leq &\int\limits_{\mathbb{T}}\left\vert \overline{S}_{m+1}\left(
f;x,y\right) \right\vert dx \\
&\leq &c_{1}\int\limits_{\mathbb{T}}\left\vert f\left( x,y\right)
\right\vert \log ^{+}\left\vert f\left( x,y\right) \right\vert dx+c_{2},
\end{eqnarray*}%
Consequently,
\begin{eqnarray}
&&\iint\limits_{\mathbb{T}^{2}}\left( \frac{1}{l_{n}l_{m}}%
\sum\limits_{i=0}^{n}\sum\limits_{j=0}^{m}\frac{\left\vert I_{15}\left(
i,m;x,y\right) \right\vert }{\left( i+1\right) \left( j+1\right) }\right)
^{p}dxdy  \label{I15} \\
&\leq &c_{1}\iint\limits_{\mathbb{T}^{2}}\left\vert f\left( x,y\right)
\right\vert \log ^{+}\left\vert f\left( x,y\right) \right\vert dxdy+c_{2}.
\notag
\end{eqnarray}

Combining (\ref{I1-I5})-(\ref{I15}) we get%
\begin{eqnarray}
&&\iint\limits_{\mathbb{T}^{2}}\left( \frac{1}{l_{n}l_{m}}%
\sum\limits_{i=0}^{n}\sum\limits_{j=0}^{m}\frac{\left\vert I_{1}\left(
i,j;x,y\right) \right\vert }{\left( i+1\right) \left( j+1\right) }\right)
^{p}dxdy  \label{I1} \\
&\leq &c_{1}\iint\limits_{\mathbb{T}^{2}}\left\vert f\left( x,y\right)
\right\vert \log ^{+}\left\vert f\left( x,y\right) \right\vert dxdy+c_{2}.
\notag
\end{eqnarray}

Analogously, we can prove%
\begin{eqnarray}
&&\iint\limits_{\mathbb{T}^{2}}\left( \frac{1}{l_{n}l_{m}}%
\sum\limits_{i=0}^{n}\sum\limits_{j=0}^{m}\frac{\left\vert I_{s}\left(
i,j;x,y\right) \right\vert }{\left( i+1\right) \left( j+1\right) }\right)
^{p}dxdy  \label{I2-I4} \\
&\leq &c_{1}\iint\limits_{\mathbb{T}^{2}}\left\vert f\left( x,y\right)
\right\vert \log ^{+}\left\vert f\left( x,y\right) \right\vert
dxdy+c_{2},s=2,3,4,  \notag
\end{eqnarray}%
\begin{eqnarray}
&&\left( \iint\limits_{\mathbb{T}^{2}}\left( \frac{1}{l_{n}l_{m}}%
\sum\limits_{i=0}^{n}\sum\limits_{j=0}^{m}\frac{\left\vert \overline{S}%
_{n+1}\left( S_{m-j}\left( f,y\right) ;x\right) \right\vert }{i+1}\right)
^{p}dxdy\right) ^{1/p}  \label{S_nm-j} \\
&=&\left( \iint\limits_{\mathbb{T}^{2}}\left( \frac{1}{l_{n}l_{m}}%
\sum\limits_{i=0}^{n}\sum\limits_{j=0}^{m}\frac{\left\vert S_{m-j}\left(
\overline{S}_{n+1}\left( f,x\right) ;y\right) \right\vert }{i+1}\right)
^{p}dxdy\right) ^{1/p}  \notag \\
&=&\left( \iint\limits_{\mathbb{T}^{2}}\left( \frac{1}{l_{m}}%
\sum\limits_{j=0}^{m}\frac{\left\vert S_{m-j}\left( \overline{S}_{n+1}\left(
f,x\right) ;y\right) \right\vert }{i+1}\right) ^{p}dxdy\right) ^{1/p}  \notag
\\
&\leq &\leq c_{1}\iint\limits_{\mathbb{T}^{2}}\left\vert f\left( x,y\right)
\right\vert \log ^{+}\left\vert f\left( x,y\right) \right\vert dxdy+c_{2}.
\notag
\end{eqnarray}

Combining (\ref{S_n-im-j}), (\ref{I1}), (\ref{I2-I4}) and (\ref{S_nm-j}) we
complete the proof of Theorem \ref{logest}.
\end{proof}

By the density of polynomials and by virtue of standard arguments \cite{Zy}
we can see the validity of Theorem \ref{theoremconv}

\end{document}